\documentclass[12pt,twoside]{amsart}
\usepackage{amssymb}
\usepackage{verbatim}
\usepackage{amsmath}
\usepackage{bm}
\usepackage{a4wide}
\usepackage{xcolor}
\usepackage{tikz}
\usepackage[latin1]{inputenc}
\usepackage[T1]{fontenc}
\usepackage{times}
\usepackage{amssymb,latexsym}
\usepackage{enumerate}
\usepackage{pict2e}

\makeatletter
\DeclareRobustCommand{\intprod}{%
  \mathbin{\mathpalette\int@prod{(0.1,0)(0.9,0)(0.9,0.8)}}%
}
\DeclareRobustCommand{\intprodr}{%
  \mathbin{\mathpalette\int@prod{(0.1,0.8)(0.1,0)(0.9,0)}}}

\newcommand{\int@prod}[2]{%
  \begingroup
  \sbox\z@{$\m@th#1+$}%
  \setlength\unitlength{\wd\z@}%
  \begin{picture}(1,1)
  \roundcap
  \polyline#2
  \end{picture}%
  \endgroup
}
\makeatother

\makeatletter
\newcommand{\sumprime}{\if@display\sideset{}{'}\sum%
            \else\sum'\fi}
\makeatother

\begin{document}

\numberwithin{equation}{section}

\newtheorem{theorem}{Theorem}[section]
\newtheorem{proposition}[theorem]{Proposition}
\newtheorem{conjecture}[theorem]{Conjecture}
\def\theconjecture{\unskip}
\newtheorem{corollary}[theorem]{Corollary}
\newtheorem{lemma}[theorem]{Lemma}
\newtheorem{observation}[theorem]{Observation}
\newtheorem{definition}{Definition}
\numberwithin{definition}{section} 
\newtheorem{remark}{Remark}
\def\theremark{\unskip}
\newtheorem{kl}{Key Lemma}
\def\thekl{\unskip}
\newtheorem{question}{Question}
\def\thequestion{\unskip}
\newtheorem{example}{Example}
\def\theexample{\unskip}
\newtheorem{problem}{Problem}

\thanks{Supported by National Natural Science Foundation of China, No. 11771089}

\address{School of Mathematical Sciences, Fudan University, Shanghai, 200433, China}

\email{boychen@fudan.edu.cn}

\address{School of Mathematical Sciences, Fudan University, Shanghai, 200433, China}

\email{ypxiong18@fudan.edu.cn}

\title{A Psh Hopf Lemma for Domains with Cusp Conditions}

\author{Bo-Yong Chen and Yuanpu Xiong}
\date{}

\begin{abstract}
We obtain a psh Hopf lemma for domains satisfying certain cusp conditions by using a sharp estimate for the Green function of a planar cusp along the axis. As an application, we obtain a negative psh exhaustion function with certain global growth estimate on a pseudoconvex domain with H\"{o}lder boundary.
\end{abstract}
\maketitle
\section{Introduction}
Let $\Omega\subset\mathbb{R}^n$ be a domain satisfying the inner ball condition (e.g., a domain with $C^2$ boundary). The classical Hopf lemma asserts that if $u$ is a negative harmonic function on $\Omega$ with $u(x_0)=0$ for some $x_0\in\partial\Omega$, then $\partial{u}(x_0)/\partial\nu<0$, where $\nu$ is the inner normal vector field of $\partial\Omega$. In particular, we have $u\lesssim-\delta_\Omega$, where $\delta_\Omega$ is the boundary distance. Sometimes, such an inequality is also called a Hopf lemma in literature. Indeed, we only need $u\lesssim-\delta_\Omega^{1/\alpha}$ ($0<\alpha<1$) for some applications (cf. \cite{FS}), which is available for domains satisfying cone conditions (cf. \cite{Mercer}, \cite{Oddson}).

In this paper, we obtain a Hopf lemma for plurisubharmonic (psh) functions on domains satisfying cusp conditions. A domain of the form
\[
\Gamma=\Gamma(C,\alpha)=\left\{z\in\mathbb{C}^n;\ \mathrm{Re}\,z_n>C(|z'|^2+(\mathrm{Im}\,z_n)^2)^{\alpha/2}\right\}
\]
is called a $(C,\alpha)$-cusp along the $\mathrm{Im}\,z_n$-axis and with vertex at the origin. By a finite $(C,\alpha,r)$-cusp we mean an intersection of $\Gamma$ with a ball centered at the vertex with radius $r$. Roughly speaking, a domain is said to satisfy the $(C,\alpha,r)$-cusp condition if for any $p\in\partial\Omega$, one can transplant a finite $(C,\alpha,r)$-cusp inside $\Omega$ with vertex $p$ (through translations and rotations). We refer to Definition \ref{def:cusp} for more details.

The main result of the paper is the following

\begin{theorem}[Hopf lemma]\label{th:Hopf}
Let $\Omega\subset\subset\mathbb{C}^n$ be a domain satisfying the $(C,\alpha,r)$-cusp condition. Then
\begin{equation}\label{eq:Hopf}
\rho\lesssim-\exp\left(-\frac{A}{\delta_\Omega^{1/\alpha-1}}\right)
\end{equation}
for every $\rho\in{PSH}^-(\Omega)$, where $PSH^-(\Omega)$ denotes the set of negative psh functions on $\Omega$, and
\begin{equation}\label{eq:A}
A:=\frac{\pi{C^{1/\alpha}}}{2(1/\alpha-1)}.
\end{equation}
\end{theorem}

Usually, Hopf lemmas for subharmonic functions are proved by constructing barrier functions for the complement of the closure of the domain (see, e.g., \cite{Miller67}, \cite{Miller71}, \cite{Oddson}). This breaks down in our case. Indeed, one can apply Wiener's criterion to show that a closed cusp $\overline{\Gamma}$ is thin at the vertex for $0<\alpha<1$ and $n\geq2$ (see Appendix). That is, there does not exist a barrier function at the vertex of $\Gamma$ for $\mathbb{C}^n\setminus\overline{\Gamma}$, which clearly satisfies certain cusp condition. On the other hand, a planar cusp is not thin at the vertex. The proof of Theorem \ref{th:Hopf} is based on the following optimal estimate of the Green function on a planar cusp along the axis, which is of independent interest.

\begin{theorem}\label{th:estimate of Green function cusp}
Let
\begin{equation}\label{eq:planar cusp}
\Gamma:=\{z=x+iy;\ x>C|y|^\alpha\},
\end{equation}
where $C>0$ and $0<\alpha<1$, and $\Delta_R$ be the disk centered at 0 with radius $R$. For $a:=R/2$ and $0<t\ll1$, the (negative) Green function of \/$\Gamma\cap\Delta_R$ satisfies
\[
g_{\Gamma\cap\Delta_R}(t,a)\approx-\exp\left(-\frac{A}{t^{1/\alpha-1}}\right).
\]
\end{theorem}

The key ingredient of the proof of Theorem \ref{th:estimate of Green function cusp} is an explicit construction of a comparison function of a planar cusp at the vertex by using conformal mappings and a result of Li-Nirenberg \cite{LiNirenberg}.

This paper is motivated by a recent work \cite{Chen20} of the first author, where the local order of hyperconvexity is obtained for bounded pseudoconvex domains with H\"{o}lder boundaries (written as H\"{o}lder domains for short). This together with a result of Kerzman and Rosay \cite{KR} imply the hyperconvexity of such domains. It remains a question whether there exists a negative psh exhaustion function with certain global estimate on bounded pseudoconvex H\"{o}lder domains.

We shall give a partial answer to this question by using Theorem \ref{th:Hopf}. Let $\Omega$ be a bounded pseudoconvex H\"{o}lder domain, which naturally satisfies certain $(C,\alpha,r)$-cusp condtion (see Proposition \ref{prop:Holder boundary domain}). Theorem \ref{th:Hopf} and the main result in \cite{Chen20} yield
\begin{equation}\label{eq:local estimate2}
\psi(-\delta_\Omega(z))\leq\rho_j(z)\leq\varphi(-\delta_\Omega(z))
\end{equation}
where $\psi(t):=-C_1(-\log(-t))^{-\beta}$ and $\varphi(t):=-C_2\exp(-A/(-t)^{1/\alpha-1})$ for certain positive constants $\alpha,\beta,C_1,C_2,M$. Let $\{\alpha_\nu\}$ be an increasing sequence with $\alpha_\nu<0$, $\alpha_1$ sufficiently close to 0, and
\begin{equation}\label{eq:definition of alpha nu}
\alpha_{\nu+1}=\psi^{-1}\big(\varphi(\alpha_\nu)/2\big).
\end{equation}
We define
\[
\lambda(t):=\max\{\nu\in\mathbb{Z}^+;\ \alpha_\nu\leq-t\},\ \ \ t>0.
\]
Since $\lim_{\nu\rightarrow\infty}\alpha_\nu=0$, it follows that $\lambda(t)\rightarrow+\infty$ as $t\rightarrow0+$. It is easy to see that $\lambda$ is not an elementary function. The method of Coltoiu-Mihalche \cite{CM} together with Lemma 2.1 in \cite{Chen20} yield the following

\begin{theorem}\label{th:Holder Global}
Let $\Omega\subset\mathbb{C}^n$ be a bounded pseudoconvex H\"{o}lder domain and $\overline{B}$ be a closed ball in $\Omega$. Then there exists a constant $\varepsilon_0>0$ depending only on $\Omega$ and $\overline{B}$ such that
\begin{equation}\label{eq:Holder global}
\varrho_{\Omega,\overline{B}}\gtrsim-e^{-\varepsilon_0\lambda(\delta_\Omega)}.
\end{equation}
Here $\varrho_{\Omega,\overline{B}}$ denotes the relative extremal function of $\Omega$ with respect to $\overline{B}$, i.e.,
\[
\varrho_{\Omega,\overline{B}}(z):=\sup\{u(z);\ u\in{PSH^-(\Omega)},\ u|_{\overline{B}}\leq-1\}.
\]
\end{theorem}

It is well-known that $\varrho_{\Omega,\overline{B}}\in{PSH^-(\Omega)}$ and is maximal in $\Omega\setminus\overline{B}$ if $\Omega$ is hyperconvex. Moreover, under the same condition, $\varrho_{\Omega,\overline{B}}$ is continuous on $\Omega$ and $\varrho_{\Omega,\overline{B}}(z)\rightarrow0$ when $z\rightarrow\partial\Omega$. (cf. \cite{Blocki}, \S 3.1)

\section{Proof of Theorem \ref{th:estimate of Green function cusp}}
Consider the planar $(C,\alpha)$-cusp given in \eqref{eq:planar cusp}. For $-C^{-1/\alpha}\leq c \leq C^{-1/\alpha}$, we define $\gamma_c$ to be the curve
\[
\gamma_c(s)=s+ics^{1/\alpha},\ \ \ s>0,
\]
so that
\[
\Gamma=\{\gamma_c(s);\ -C^{-1/\alpha}<c<C^{-1/\alpha},\ s>0\}.
\]
For fixed $c$, we write $\gamma_c(s)=r(s)e^{i\theta(s)}$, where
\begin{equation}\label{eq:r and theta}
r(s)=\left(s^2+c^2s^{2/\alpha}\right)^{1/2},\ \ \ \theta(s)=\arctan\left(cs^{1/\alpha-1}\right).
\end{equation}

Let $F$ be the holomorphic function obtained by a single-valued branch of
\begin{equation}\label{eq:F}
\exp\left(-\frac{A}{z^{1/\alpha-1}}\right),
\end{equation}
where $F(1)=e^{-A}$ and $A$ is given in \eqref{eq:A}. We also write $F(\gamma_c(s))=\widetilde{r}(s)e^{i\widetilde{\theta}(s)}$. It follows that
\begin{eqnarray*}
\widetilde{r}(s)e^{i\widetilde{\theta}(s)} &=& \exp\left(-\frac{A}{r(s)^{1/\alpha-1}}e^{-i(1/\alpha-1)\theta(s)}\right)\\
&=& \exp\left(-\frac{A\cos\big((1/\alpha-1)\theta(s)\big)}{r(s)^{1/\alpha-1}}+i\,\frac{A\sin\big((1/\alpha-1)\theta(s)\big)}{r(s)^{1/\alpha-1}}\right).
\end{eqnarray*}
Hence
\begin{equation}\label{eq:r tilde}
\widetilde{r}(s)=\exp\left(-\frac{A\cos\big((1/\alpha-1)\theta(s)\big)}{r(s)^{1/\alpha-1}}\right)
\end{equation}
and
\begin{equation}\label{eq:theta tilde}
\widetilde{\theta}(s)=\frac{A\sin\big((1/\alpha-1)\theta(s)\big)}{r(s)^{1/\alpha-1}}.
\end{equation}
Since $r(s)\geq{s}$ and $|\theta(s)|\leq{|c|s^{1/\alpha-1}}$ in view of \eqref{eq:r and theta}, we infer from \eqref{eq:A} that
\begin{equation}\label{eq:choice of A}
\left|\widetilde{\theta}(s)\right|\leq{A(1/\alpha-1)|c|}\leq{A}(1/\alpha-1)C^{-1/\alpha}=\frac{\pi}{2}.
\end{equation}
Moreover, since $\sin{t}\sim{t}$ and $\arctan{t}\sim{t}$ as $t\rightarrow0$, we have $\widetilde{\theta}(s)\rightarrow\pm\pi/2$ as $s\rightarrow0$ and $c=\pm{C^{-1/\alpha}}$ in view of \eqref{eq:r and theta}.

\begin{lemma}\label{lm:F}
There exists a constant $R>0$ depending only on $C$ and $\alpha$ such that the holomorphic function $F$, given in \eqref{eq:F}, maps $\Gamma\cap\Delta_R$ conformally to a domain $D\subset\mathbb{C}$.
\end{lemma}
\begin{proof}
It suffices to show that $F$ is injective on $\Gamma\cap\Delta_R$ for some sufficiently small $R$. We first verify the injectivity of $G(z):=-A/z^{1/\alpha-1}$. Given $z_1=r_1e^{i\theta_1}$ and $z_2=r_2e^{i\theta_2}$, $G(z_1)=G(z_2)$ if and only if $r_1=r_2$ and $(1/\alpha-1)(\theta_1-\theta_2)=2k\pi$ for some integer $k$. Since
\[|\theta(s)|\leq{|c|s^{1/\alpha-1}}\leq{C}^{1/\alpha}s^{1/\alpha-1}\]
in view of \eqref{eq:r and theta}, we have $|\theta_1-\theta_2|<2\pi/(1/\alpha-1)$ for $z_1,z_2\in\Gamma\cap\Delta_R$, provided that
\[C^{1/\alpha}R^{1/\alpha-1}<\frac{\pi}{(1/\alpha-1)}.\]
Thus $G$ is injective on $\Gamma\cap\Delta_R$. Moreover, for $z=\gamma_c(s)\in\Gamma$, we have $\mathrm{Im}\,G(z)=\widetilde{\theta}(s)\in(-\pi/2,\pi/2)$ in view of \eqref{eq:choice of A}. Since the exponential mapping $w\mapsto{e^w}$ is injective on the strip $\{|\mathrm{Im}\,w|<\pi/2\}$, we conclude that $F=e^G$ is injective on $\Gamma\cap\Delta_R$.
\end{proof}

The domain $D$ is also symmetric about the $\widetilde{x}$-axis and $\partial{D}$ is more regular than $\partial\Gamma$. In a neighbourhood of $0\in\partial{D}$, the curve $\partial{D}$ is precisely $F\circ\gamma_c$ with $c=\pm{C^{-1/\alpha}}$. Since $\widetilde{\theta}(s)\rightarrow\pm\pi/2$ as $s\rightarrow0$ and $c=\pm{C^{-1/\alpha}}$, $\partial{D}$ is tangent to the $y$-axis at $0$. To obtain an estimate for the Green function of $D$ near $0\in\partial{D}$, we need more information about the regularity of $\partial{D}$ near $0$. Choose a neighbourhood $U\ni0$ such that
\[D\cap{U}=\{z=x+iy\in{U};\ x>h(|y|),\ |y|<\varepsilon_0\},\]
where $h$ is a real function on $[0,\varepsilon_0)$ and $0<\varepsilon_0\ll1$ is a constant depending only on $C$ and $\alpha$. Note that $h$ is smooth on $(0,\varepsilon_0)$ and differentiable at 0 with $h(0)=h'(0)=0$. Now we fix $c=C^{-1/\alpha}$ and set
\[\widetilde{x}(s)=\widetilde{r}(s)\cos\widetilde{\theta}(s),\ \ \ \widetilde{y}(s)=\widetilde{r}(s)\sin\widetilde{\theta}(s).\]
The function $h$ is given by the equation
\[\widetilde{x}(s)=h\left(\widetilde{y}(s)\right),\ \ \ 0<s\ll1.\]
Thus
\begin{equation}\label{eq:h 1}
h\left(\widetilde{y}(s)\right)=\frac{\widetilde{y}(s)}{\tan\widetilde{\theta}(s)}\sim\widetilde{y}(s)\cdot\left(\frac{\pi}{2}-\widetilde{\theta}(s)\right),\ \ \ s\rightarrow0.
\end{equation}
Moreover, \eqref{eq:theta tilde} together with \eqref{eq:r and theta} imply
\begin{eqnarray}\label{eq:theta tilde expansion}
\widetilde{\theta}(s) &=& \frac{A\sin\big((1/\alpha-1)\arctan(C^{-1/\alpha}s^{1/\alpha-1})\big)}{s^{1/\alpha-1}(1+s^{2/\alpha-2})^{(1/\alpha-1)/2}}\nonumber\\
&=& \frac{A\sin\big((1/\alpha-1)(C^{1/\alpha}s^{1/\alpha-1}-\frac{1}{3}C^{-3/\alpha}s^{3/\alpha-3}+o(s^{3/\alpha-3}))\big)}{s^{1/\alpha-1}+\frac{1}{2}(1/\alpha-1)s^{3/\alpha-3}+o(s^{3/\alpha-3})}\\
&=& \frac{\pi}{2}-A_1s^{2/\alpha-2}+o(s^{2/\alpha-2}),\ \ \ s\rightarrow0\nonumber
\end{eqnarray}
for some constant $A_1=A_1(C,\alpha)>0$. It follows from \eqref{eq:r and theta}, \eqref{eq:r tilde} and \eqref{eq:theta tilde expansion} that
\begin{eqnarray}\label{eq:log y tilde}
\log\frac{1}{\widetilde{y}(s)}&=&\log\frac{1}{\widetilde{r}(s)}-\log\sin\widetilde{\theta}(s)\nonumber\\
&=& \frac{A\cos\big((1/\alpha-1)\theta(s)\big)}{r(s)^{1/\alpha-1}}-\log\sin\widetilde{\theta}(s)\\
&\sim& \frac{A}{s^{1/\alpha-1}},\ \ \ s\rightarrow0.\nonumber
\end{eqnarray}
By \eqref{eq:h 1}, \eqref{eq:theta tilde expansion} and \eqref{eq:log y tilde}, we finally obtain
\begin{equation}\label{eq:h}
h(|y|)\sim{A_2}|y|\left(\log\frac{1}{|y|}\right)^{-2},\ \ \ y\rightarrow0,
\end{equation}
where $A_2=A_1A^2$.

Recall that the (negative) Green function of a planar domain $\Omega$ can be defined as
\[
g_\Omega(z,a):=\sup\left\{\varphi\in{SH^-(\Omega)};\ \varphi(z)-\log\frac{1}{|z-a|}=O(1)\right\},
\]
where $SH^-(\Omega)$ denotes the set of negative subharmonic functions on $\Omega$. It is well-known that $g_\Omega(\cdot,a)$ is a subharmonic function on $\Omega$ and is harmonic on $\Omega\setminus\{a\}$.

\begin{proof}[Proof of Theorem \ref{th:estimate of Green function cusp}]
By Lemma \ref{lm:F}, we have
\begin{equation}\label{eq:Green invariant}
g_{\Gamma\cap\Delta_R}(t,a)=g_D(F(t),F(a)).
\end{equation}
It suffices to estimate $g_D(x,F(a))$ when $x\rightarrow0+$ and $x\in\mathbb{R}$.

From \eqref{eq:h}, there is a constant $B>0$ such that
\begin{equation}\label{eq:hB}
h(u)<Bu\left(\log\frac{1}{u}\right)^{-2}=:h_B(u),\ \ \ 0<u<u_0<1.
\end{equation}
Here, the constant $u_0$ depends only on $C$ and $\alpha$. Let
\[
D_B:=\{z=x+iy;\ x>h_B(|y|),\ -u_0<y<u_0\}.
\]
It follows from \eqref{eq:hB} that $D_B\cap\Delta_r\subset{D}$ for some $0<r\ll1$. We may choose $r=r(C,\alpha)$ so small that $F(a)\notin\Delta_r$. Thus $g_D(\cdot,F(a))$ is a harmonic function on $D\cap\Delta_r$.

The boundary of $D_B\cap\Delta_r$ can be devided into the following two parts:
\[
E_1:=\overline{D}_B\cap\{|z|=r\},\ \ \ E_2:=\{z=x+iy;\ x=h_B(|y|),\ |z|<r\}.
\]
By \eqref{eq:hB}, $E_1$ is a compact subset in $D$, while $E_2$ satisfies the following conditions:

(1) $h_B$ is smooth on $(0,u_0)$ and $C^1$ on $[0,u_0)$;

(2) $\int^\varepsilon_0h_B(u)/u^2du<\infty$ for $0<\varepsilon<\min\{u_0,1\}$;

(3) $h_B(0)=h_B'(0)=0$ and $h_B'(u)\geq0$;

(4) $h_B''(u)+h_B'(u)/u$ is nonincreasing.

\noindent It follows from the proof of Theorem 1 of \cite{LiNirenberg} that
\[
\varphi(z):=x+2h_B(x)+2x\int^x_0\frac{h_B(u)}{u^2}du-2h_B(|z|),\ \ \ (z=x+iy)
\]
is a subharmonic function on $D_B\cap\Delta_r$ with $\varphi\leq0$ on $E_2$ and $\varphi(x)>0$ for $x>0$. The constant
\[M:=\frac{\sup_{E_1}\varphi}{\inf_{E_1}(-g_D)}\]
is positive and depends only on $C$ and $\alpha$. Thus
\begin{equation}\label{eq:maximal Green}
-g_D(z,F(a))\geq\frac{\varphi(z)}{M},\ \ \ z\in{E_1}.
\end{equation}
The same inequality clearly holds for $z\in{E_2}$ since $g_D\leq0$. Thus \eqref{eq:maximal Green} holds for all $z\in{D_B\cap\Delta_r}$ in view of the maximum principle. In particular,
\[
g_D(x,F(a))\lesssim-\varphi(x)\leq-x,\ \ \ 0<x\ll1.
\]
On the other hand, by comparing $g_D(\cdot,F(a))$ with the negative harmonic function $z=x+iy\mapsto-x$ on $D$, we have
\[g_D(x,F(a))\gtrsim-x,\ \ \ 0<x\ll1.\]
The conclusion follows immediately from \eqref{eq:Green invariant}.
\end{proof}

\section{The Cusp Condition}
For a unit vector $v\in\mathbb{C}^n$, we consider the half space
\[
H_v:=\{z\in\mathbb{C}^n;\ \mathrm{Re}\,\langle{z,v}\rangle>0\}.
\]
Let $\pi_v$ be the orthogonal projection $H_v\rightarrow\partial{H_v}$. A $(C,\alpha)$-cusp with axis $v$ and vertex $p$ is defined to be
\[
\Gamma(p,v,C,\alpha):=\{z\in{H_v};\ \mathrm{Re}\,\langle{z-p,v}\rangle>{C}|\pi_v(z-p)|^\alpha\}.
\]

\begin{definition}\label{def:cusp}
Let $\Omega\subset\mathbb{C}^n$ be a bounded domain. Let $B_r(p)$ be the ball with center $p$ and radius $r$. We say that $\Omega$ satisfies the $(C,\alpha,r)$-cusp condition if there is some $r>0$, such that every $z\in\Omega$ sufficiently close to $\partial\Omega$ lies on the axis of a $(C,\alpha)$-cusp $\Gamma(p,v,C,\alpha)$ with

(1) $p\in\partial\Omega$;

(2) $\Gamma(p,v,C,\alpha)\cap{B_r(p)}\subset\Omega$;

(3) $|z-p|<r/2$.
\end{definition}

The condition (3) implies that $z$ is not very close to $\partial{B_r(p)}$ so that $|z-p|<\delta_{\partial{B_r(p)}}(z)$. Hence
\begin{equation}\label{eq:finite cusp distance}
\delta_{\Gamma\cap{B_r(p)}}(z)=\delta_\Gamma(z),
\end{equation}
where $\Gamma:=\Gamma(p,v,C,\alpha)$.

\begin{lemma}\label{lm:cusp}
For $z:=p+tv\in\Gamma(p,v,C,\alpha)$ $(0<t\ll1)$, we have
\[
\delta_\Gamma(z)\approx{t}^{1/\alpha}.
\]
\end{lemma}
\begin{proof}
We may assume that $p=0$ and $v=(0,\cdots,0,1)$. Set $z=(0,\cdots,0,\mathrm{Re}\,z_n)$ (i.e., $t=\mathrm{Re}\,z_n$). For any $z^*\in\partial\Gamma(p,v,C,\alpha)$, we define
\[w^*=(z^*_1,\cdots,z^*_{n-1},i\mathrm{Im}\,z^*_n)\in\mathbb{C}^{n-1}\times\mathbb{R}\]
so that $\mathrm{Re}\,z_n^*=C|w^*|^\alpha$. In particular, if we take $z_0^*=(C^{-1/\alpha}t^{1/\alpha},0,\cdots,0,t)\in\partial\Gamma(p,v,C,\alpha)$, then
\[\delta_\Gamma(z)\leq|z-z_0^*|=C^{-1/\alpha}|z-p|^{1/\alpha}.\]
On the other hand, we have
\[
|z-z^*|^2=|w^*|^2+|\mathrm{Re}\,(z_n-z_n^*)|^2.
\]
We shall devide the argument into the following two cases:

(1) If $\mathrm{Re}\,z^*_n\leq\frac{1}{2}\mathrm{Re}\,z_n$, then $\mathrm{Re}\,(z_n-z^*_n)\geq\frac{1}{2}\mathrm{Re}\,z_n>0$, so that
\[
|z-z^*|\geq|\mathrm{Re}\,(z_n-z^*_n)|\geq\frac{1}{2}\mathrm{Re}\,z_n=\frac{1}{2}t\geq\frac{1}{2}t^{1/\alpha};
\]

(2) If $\mathrm{Re}\,z^*_n>\frac{1}{2}\mathrm{Re}\,z_n$, then we have
\begin{eqnarray*}
|z-z^*|\geq|w^*| &=& \left(\frac{1}{C}\right)^{1/\alpha}\left(\mathrm{Re}\,z^*_n\right)^{1/\alpha}>\left(\frac{1}{2C}\right)^{1/\alpha}\left(\mathrm{Re}\,z_n\right)^{1/\alpha}\\
&=& \left(\frac{1}{2C}\right)^{1/\alpha}t^{1/\alpha}.
\end{eqnarray*}
Hence $\delta_\Gamma(z)\geq{Bt^{1/\alpha}}$ with $B:=\min\{1/2,(2C)^{-1/\alpha}\}$.
\end{proof}

\begin{remark}
Let $\Omega$ be a domain satisfying the $(C,\alpha,r)$-cusp condition and $z\in\Omega$ satisfying the conditions in Definition \ref{def:cusp}. Lemma \ref{lm:cusp} together with \eqref{eq:finite cusp distance} imply that
\[
|z-p|\lesssim\delta_\Gamma(z)^\alpha=\delta_{\Gamma\cap{B_r(p)}}(z)^\alpha\leq\delta_\Omega(z)^\alpha.
\]
That is, if $z$ is sufficiently close to the boundary, then $|z-p|$ is also sufficiently small.
\end{remark}

Recall that $\Omega$ is a bounded H\"{o}lder domain if $\partial\Omega$ is locally the graph of a H\"{o}lder continuous function. More precisely, there exist

(1) a finite open covering $\{V_j\}$ of $\partial\Omega$;

(2) $p_j\in{V_j}$;

(3) a unit vector $v_j$;

(4) a neighbourhood $V'_j$ of $0$ in $\partial{H_{v_j}}$  with $\pi_{v_j}(w-p_j)\in{V'_j}$ for all $w\in{V_j}$;

(5) a H\"{o}lder continuous function $h_j$ of order $\alpha_j$ on $V'_j$ with $h_j(0)=0$, such that
\[
\Omega\cap{V_j}=\{w\in{V_j};\ \mathrm{Re}\,\langle{w-p_j,v_j}\rangle>h_j(\pi_{v_j}(w-p_j))\}.
\]

\begin{proposition}\label{prop:Holder boundary domain}
If $\Omega\subset\mathbb{C}^n$ is a bounded H\"{o}lder domain, then there exist constants $C,\alpha$ and $r$ such that $\Omega$ satisfies the $(C,\alpha,r)$-cusp condition. Moreover precisely, $\alpha=\min\{\alpha_j\}$, where $\alpha_j$ is given as above.
\end{proposition}
\begin{proof}
Let $V_j,p_j,v_j,V'_j,h_j$ be as above. Suppose that
\begin{equation}\label{eq:Holder}
|h_j(x_1)-h_j(x_2)|\leq{C_j}|x_1-x_2|^{\alpha_j},\ \ \ x_1,x_2\in{V'_j}
\end{equation}
for some $C_j>0$. For any $w\in{V_j}$, let $w^*\in{\partial\Omega}$ be the point with
\[
\pi_{v_j}(w-p_j)=\pi_{v_j}(w^*-p_j).
\]
That is, $w=w^*+tv_j$ for some $t$. We may take another covering $\{U_j\}$ of $\partial\Omega$ with $p_j\in{U_j}\subset\subset{V_j}$, such that $w^*\in{U_j\cap\partial\Omega}$ and $|w-w^*|<r/2$ for some $0<r\ll1$ whenever $w\in{U_j}$. Moreover, we may take $r$ so small that $B_r(x)\subset{V_j}$ for all $x\in{U_j}\cap\partial\Omega$. Set
\[
U_j^+:=U_j\cap\Omega.
\]
Then any $z\in\Omega$ sufficiently close to $\partial\Omega$ lies in some $U_j^+$. For $p:=z^*$, it follows from \eqref{eq:Holder} and the definition of $U_j$ that $z\in\Gamma(p,v_j,C_j,\alpha_j)\cap{B_r(p)}$ and $|z-p|<r/2$. Moreover, for any $w\in\Gamma(p,v_j,C_j,\alpha_j)\cap{B_r(p)}$, we have $w\in{V_j}$ and
\begin{eqnarray*}
\mathrm{Re}\,\langle{w-p,v_j}\rangle &>& C_j|\pi_{v_j}(w-p)|^\alpha = C_j|\pi_{v_j}(w-p_j)-\pi_{v_j}(p-p_j)|^\alpha\\
&\geq& h(\pi_{v_j}(w-p_j))-h(\pi_{v_j}(p-p_j)).
\end{eqnarray*}
Since $p\in\partial\Omega$, we have
\[\mathrm{Re}\,\langle{p-p_j,v_j}\rangle=h(\pi_{v_j}(p-p_j)).\]
Hence $\mathrm{Re}\,\langle{w-p_j,v_j}\rangle>h(\pi_{v_j}(w-p_j))$, i.e.,
\[
\Gamma(p,v_j,C_j,\alpha_j)\cap{B_r(p)}\subset\Omega\cap{V_j}\subset\Omega.
\]
It suffices to take $\alpha=\min\{\alpha_j\}$ and $C=\max\{C_j\}$.
\end{proof}

\section{Proof of Theorem \ref{th:Hopf}}
The following lemma is essentially known, but we shall provide a proof for the sake of completeness.
\begin{lemma}\label{lm:Green function and nonpositive subharmonic functions}
Let $\Omega\subset\mathbb{C}$ be a bounded domain and $\rho\in{SH^-(\Omega)}$. If $\Delta_{R_1}(a)\subset\subset\Omega\subset\subset{\Delta_{R_2}(a)}$, then
\[
\rho(z)\leq\frac{\inf_{\partial{\Delta_{R_1}(a)}}(-\rho)}{\log({R_2}/{R_1})}\cdot g_\Omega(z,a)
\]
in a neighbourhood of $\partial{\Omega}$.
\end{lemma}
\begin{proof}
Let $\{\Omega_m\}$ be a sequence of domains with smooth boundaries such that $\Omega_m\subset\subset\Omega_{m+1}$ and $\bigcup{\Omega_m}=\Omega$. Then $g_{\Omega_m}\searrow{g_\Omega}$ when $m\rightarrow\infty$. We may assume that $\Delta_{R_1}(a)\subset\subset{\Omega_m}\subset\subset{\Delta_{R_2}(a)}$. Thus
\[
g_{\Omega_m}(z,a)\geq\log\frac{|z-a|}{R_2},\ \ \ z\in\Omega_m.
\]
In particular,
\[
g_{\Omega_m}(z,a)\geq-\log\frac{R_2}{R_1},\ \ \ z\in\partial{\Delta}_{R_1}(a).
\]
For $C_0:=\inf_{\partial{\Delta_{R_1}(a)}}(-\rho)=-\sup_{\partial{\Delta_{R_1}(a)}}\rho$, we have
\[
\frac{\log(R_2/R_1)}{C_0}\rho\leq-\log\frac{R_2}{R_1},\ \ \ z\in\partial{\Delta}_{R_1}(a).
\]
Note that $g_{\Omega_m}(\cdot,a)=0$ on $\partial{\Omega_m}$. It follows from the maximum principle that
\[
\rho(z)\leq\frac{C_0}{\log(R_2/R_1)}g_{\Omega_n}(z,a)
\]
on $\Omega_m\setminus\overline{\Delta_{R_1}(a)}$. Letting $m\rightarrow\infty$, we complete the proof.
\end{proof}

\begin{proof}[Proof of Theorem \ref{th:Hopf}]
Let $R$ be the constant in Lemma \ref{lm:F} for planar $(C,\alpha)$-cusps. We may assume that $0<R<r$. Given $z\in\Omega$ sufficiently close to $\partial\Omega$ (i.e., $\delta_\Omega(z)\ll1$), there exists a $(C,\alpha)$-cusp $\Gamma(p,v,C,\alpha)$ satisfying the conditions (1)-(3) in Definition \ref{def:cusp} with $z$ lying on the axis. By the remark after Lemma \ref{lm:cusp}, we have $|z-p|\ll1$. We may identify
\[
D_p:=\{p+tv;\ t\in\mathbb{C}\}\cap\Gamma(p,v,C,\alpha)\cap{B_R(p)}
\]
with a domain in $\mathbb{C}$. Thus $\rho|_{D_p}$ is a subharmonic function on $D_p$. Set $a=p+Rv/2$. By Lemma \ref{lm:Green function and nonpositive subharmonic functions} and Theorem \ref{th:estimate of Green function cusp}, we have
\[
\rho(z)\lesssim g_{D_p}(z,a)\leq g_{D_p\cap{B_R(p)}}(z,a)\lesssim -\exp\left(-\frac{A}{|z-p|^{1/\alpha-1}}\right).
\]
Since $|z-p|\geq\delta_\Omega(z)$, we conclude the proof.
\end{proof}

\section{Proof of Theorem \ref{th:Holder Global}}
Recall that $\{U_j\}$ is a finite covering of $\partial\Omega$, $\rho_j$ is a negative psh function on $U_j$ with $\psi(-\delta_\Omega)\leq\rho_j\leq\varphi(-\delta_\Omega)$, and $\alpha_{\nu+1}=\psi^{-1}(\varphi(\alpha_\nu)/2)$ with $\alpha_1$ sufficiently close to $0$. We set
\begin{equation}\label{eq:definition of a nu}
a_\nu=\begin{cases}
\varphi(\alpha_\nu),\ \ \ &\text{if}\ \nu\ \text{is odd;}\\
\psi(\alpha_\nu),\ \ \ &\text{if}\ \nu\ \text{is even.}
\end{cases}
\end{equation}
From \eqref{eq:definition of alpha nu} and the fact that $\varphi\geq\psi$, we obtain
\begin{equation}\label{eq:property of a nu}
a_{2\nu}=a_{2\nu-1}/2,\ \ \ a_{2\nu+1}=\varphi\circ\psi^{-1}\big(\varphi\circ\psi^{-1}(a_{2\nu})/2\big)\geq{a_{2\nu}/2}.
\end{equation}
In particular, $\{a_\nu\}$ is an increasing sequence with $a_\nu\rightarrow0-$ as $\nu\rightarrow\infty$. We consider the following convex increasing function $\tau:(-\infty,0)\rightarrow[0,+\infty)$ introduced in \cite{CM}:
\[
\tau(x)=\begin{cases}
0,\ \ \ &x\leq{a_1},\\
\nu-\sum^{\nu-1}_{k=1}a_{k+1}/a_k-x/a_\nu,\ \ \ &a_\nu\leq{x}\leq{a_{\nu+1}},
\end{cases}
\]
which satisfies
\[
\tau(a_{\nu+1})-\tau(a_\nu)=1-\frac{a_{\nu+1}}{a_\nu}<1,\ \ \ \forall\,\nu\in\mathbb{Z}^+.
\]
On the other hand, we infer from \eqref{eq:property of a nu} that $\tau(a_{\nu+1})-\tau(a_\nu)\geq1/2$. Hence
\begin{equation}\label{eq:tau(a nu)}
\tau(a_\nu)\geq\frac{\nu}{2}-c_0
\end{equation}
for some constant $c_0>0$. If $z\in\Omega\cap{U_j}\cap{U_k}$ and $\alpha_{2\nu}\leq-\delta_\Omega(z)\leq\alpha_{2\nu+2}$, then it follows from \eqref{eq:local estimate2} and \eqref{eq:definition of a nu} that
\[
\min\{\rho_j(z),\rho_k(z)\}\geq\psi(\alpha_{2\nu})=a_{2\nu}
\]
and
\[
\max\{\rho_j(z),\rho_k(z)\}\leq\varphi(\alpha_{2\nu+2})\leq\varphi(\alpha_{2\nu+3})=a_{2\nu+3}.
\]
Thus $|\tau(\rho_j)-\tau(\rho_k)|<3$ on $\Omega\cap{U_j}\cap{U_k}$. Since $\tau$ is convex, we have
\begin{equation}\label{eq:tau}
|\tau(\rho_j-\varepsilon)-\tau(\rho_k-\varepsilon)|<3
\end{equation}
on $\Omega\cap{U_j}\cap{U_k}$ for any $\varepsilon>0$. Moreover,
\begin{eqnarray}\label{eq:tau rhoj 1}
\tau(\rho_j-\varepsilon) &\leq& \tau(a_{2\nu+3}-\varepsilon)\leq\tau(a_{2\nu}-\varepsilon)+3\nonumber\\
&=& \tau(\psi(\alpha_{2\nu})-\varepsilon)+3\\
&\leq& \tau(\psi(-\delta_\Omega)-\varepsilon)+3\nonumber
\end{eqnarray}
and
\begin{eqnarray}\label{eq:tau rhoj 2}
\tau(\rho_j-\varepsilon) &\geq& \tau(a_{2\nu}-\varepsilon)\geq\tau(a_{2\nu+2}-\varepsilon)-2\nonumber\\
&=& \tau(\psi(\alpha_{2\nu+2})-\varepsilon)-2\\
&\geq& \tau(\psi(-\delta_\Omega)-\varepsilon)-2.\nonumber
\end{eqnarray}

Next we use the standard Richberg technique (compare \cite{Demailly}) to patch these $\rho_j$ together. Choose $U_j''\subset\subset{U_j'}\subset\subset{U_j}$ and $U_0\subset\subset\Omega$ such that $\overline{\Omega}\setminus{U_0}\subset\bigcup{U_j''}$. Take $\chi_j\in{C^\infty_0(U'_j)}$ with $\chi_j\equiv1$ in a neighbourhood of $\overline{U_j''}$. There are constant $M,N$ satisfying $|z|^2-M<0$ and
\begin{equation}\label{eq:N}
3\chi_j+N(|z|^2-M)\in{PSH(\mathbb{C}^n)}.
\end{equation}
Set
\[
u_{j,\varepsilon}(z):=\tau(\rho_j(z)-\varepsilon)+3\chi_j(z)-3+N(|z|^2-M).
\]
Then $u_{j,\varepsilon}$ is a plurisubharmonic function on $U_j$, and it follows from \eqref{eq:tau} that $u_{j,\varepsilon}<u_{k,\varepsilon}$ in a neighbourhood of $\Omega\cap\overline{U''_k}\cap\partial{U'_j}$. Choose $a$ so that $\sup_{U_0\cap{U_j}}\rho_j<a<0$ for all $j\geq1$ and fix $N\gg1$ such that \eqref{eq:N} holds and $\tau(a)+N(|z|^2-M)<0$ on $\overline{\Omega}$. Hence
\[
u_\varepsilon(z):=\max\{u_{j,\varepsilon}(z),\tau(a)+N(|z|^2-M)\}\in{PSH(\Omega)}
\]
when $\varepsilon\ll1$. Moreover, it follows from \eqref{eq:tau rhoj 1} and \eqref{eq:tau rhoj 2} that
\begin{equation}\label{eq:u epsilon}
\tau(\psi(-\delta_\Omega)-\varepsilon)-c_1\leq{u_\varepsilon}\leq\tau(\psi(-\delta_\Omega)-\varepsilon)+c_2
\end{equation}
for some constants $c_1,c_2>0$.

Now we shall derive a global estimate by the method in \cite{Chen20}. Fix $l\in\mathbb{Z}^+$ for a moment. We set
\[
w_\nu:=\frac{u_{-a_{2\nu+2l}}}{\tau(a_{2\nu+2l})}.
\]
For $\Omega_\nu:=\{\delta_\Omega>-\alpha_{2\nu}\}$, we infer from \eqref{eq:u epsilon} that
\[
\frac{\tau(a_{2\nu}+a_{2\nu+2l})-c_1}{\tau(a_{2\nu+2l})}\leq\sup_{\partial\Omega_\nu}w_\nu\leq\frac{\tau(a_{2\nu}+a_{2\nu+2l})+c_2}{\tau(a_{2\nu+2l})}
\]
and
\[\inf_{\Omega\setminus\Omega_{\nu+l}}w_\nu\geq\frac{\tau(\psi(\alpha_{2\nu+2l})+a_{2\nu+2l})-c_1}{\tau(a_{2\nu+2l})}=\frac{\tau(2a_{2\nu+2l})-c_1}{\tau(a_{2\nu+2l})}=\frac{\tau(a_{2\nu+2l-1})-c_1}{\tau(a_{2\nu+2l})}.\]
If we choose $l\gg1$ so that $(2l-1)/2>c_1+c_2$, then
\begin{eqnarray}\label{eq:kappa nu}
\kappa_\nu &:=& \frac{\inf_{\Omega\setminus\Omega_{\nu+l}}w_\nu-\sup_{\partial\Omega_\nu}w_\nu}{1-\sup_{\partial\Omega_\nu}w_\nu}\nonumber\\
&\geq& \frac{\tau(a_{2\nu+2l-1})-\tau(a_{2\nu}+a_{2\nu+2l})-c_1-c_2}{\tau(a_{2\nu+2l})-\tau(a_{2\nu}+a_{2\nu+2l})+c_1}\nonumber\\
&\geq& \frac{\tau(a_{2\nu+2l-1})-\tau(a_{2\nu})-c_1-c_2}{\tau(a_{2\nu+2l})-\tau(a_{2\nu-1})+c_1}\\
&\geq& \frac{(2l-1)/2-c_1-c_2}{2l+1+c_1}\nonumber\\
&\geq& c_3\nonumber
\end{eqnarray}
for some constant $c_3>0$.

The rest part of the proof is essentially parallel to \cite{Chen20}, which we still include here for the sake of completeness. Set $M_\nu:=\sup_{\Omega\setminus\Omega_\nu}(-\varrho_{\Omega,\overline{B}})$. For $z\in\partial\Omega_\nu$, we have
\begin{equation}\label{eq:maximum 1}
(1-w_\nu(z))M_\nu\geq\left[1-\sup_{\partial\Omega_\nu}w_\nu\right](-\varrho_{\Omega,\overline{B}}(z)).
\end{equation}
Since $\varrho_{\Omega,\overline{B}}(z)\rightarrow0$ when $z\rightarrow\partial\Omega$, \eqref{eq:maximum 1} also holds on $\partial\Omega_{\nu+k}$ for $k\gg1$. An equivalent statement of \eqref{eq:maximum 1} is
\begin{equation}\label{eq:maximum 2}
\varrho_{\Omega,\overline{B}}(z)\geq\frac{M_\nu}{1-\sup_{\partial\Omega_\nu}w_\nu}(w_\nu(z)-1),
\end{equation}
which actually holds for all $z\in\Omega_{\nu+k}\setminus\Omega_\nu$ by the maximal property of $\varrho_{\Omega,\overline{B}}$. Finally, letting $k\rightarrow\infty$, we conclude that \eqref{eq:maximum 2} remains valid for $z\in\Omega\setminus\Omega_{\nu+l}$, i.e.,
\[
-\varrho_{\Omega,\overline{B}}(z)\leq\frac{1-\inf_{\Omega\setminus\Omega_{\nu+l}}w_\nu}{1-\sup_{\partial\Omega_\nu}w_\nu}M_\nu=(1-\kappa_\nu)M_\nu,\ \ \ z\in\Omega\setminus\Omega_{\nu+l}.
\]
This combined with \eqref{eq:kappa nu} gives
\[
M_{\nu+l}\leq(1-c_3)M_\nu.
\]
Thus
\[-\varrho_{\Omega,\overline{B}}(z)\gtrsim(1-c_3)^{\nu/l}=\exp\left(-\frac{\nu}{l}\log\frac{1}{1-c_3}\right)\]
for $z\in\Omega\setminus\Omega_\nu$. The assertion follows immediately from the definition of $\lambda(t)$.
\qed

\section{Appendix: Thinness at the Vertex of a Closed Cusp}
Recall that a boundary point $x_0$ of a domain $\Omega\subset\mathbb{R}^m$ is regular if and only if $\Omega$ admits a barrier at $x_0$. This is also equivalent to the thinness of $\mathbb{R}^m\setminus\Omega$ at $x_0$ (cf. \cite{Landkof}, Theorem 4.8 and 5.10). Thinness can be characterized by using Wiener's criterion. For a compact set $K\subset\mathbb{R}^m$ ($m\geq3$), we define the capacity of $K$ by
\[\mathrm{Cap}\,(K):=\inf\left\{\int_{\mathbb{R}^m\setminus{K}}|\nabla\varphi|^2;\ \varphi\in{C^1_0(\mathbb{R}^m)},\ \varphi|_K\geq1\right\}.\]
Indeed, $\mathrm{Cap}\,(\cdot)$ is precisely the Newtonian capacity up to a constant multiplier (cf. \cite{Choquet}, Chapter V, 25). Wiener's criterion asserts that a closed set $E\subset\mathbb{R}^m$ ($m\geq3$) is thin at some $x_0\in\partial{E}$ if and only if
\[
\sum^\infty_{k=1}2^{k(m-2)}\mathrm{Cap}\,(E_k)<\infty,
\]
where $E_k:=E\cap\{x;\ 2^{-k-1}\leq|x-x_0|\leq2^{-k}\}$ (cf. \cite{Landkof} Theorem 5.2).

Consider the closed cusp
\[
\overline{\Gamma}:=\left\{z\in\mathbb{C}^n;\ \mathrm{Im}\,z_n\geq{C}(|z'|^2+(\mathrm{Re}\,z_n)^2)^{\alpha/2}\right\}\subset\mathbb{C}^n=\mathbb{R}^{2n}.
\]
Clearly, $\overline{\Gamma}$ is not thin at the vertex when $n=1$ since $\mathbb{C}\setminus\overline{\Gamma}$ is simply connected (cf. \cite{Ransford}, Theorem 4.2.1). Cusps can be also defined in real Euclidean spaces and every closed cusp in $\mathbb{R}^3$ is not thin at the vertex (cf. \cite{Landkof}, Chapter V, \S 1, No.3). In contrast, the following conclusion holds

\begin{proposition}\label{prop:thin}
Every closed cusp in $\mathbb{C}^n=\mathbb{R}^{2n}$ is thin at the vertex when $n\geq2$.
\end{proposition}

This result might be known. However, we still provide a proof since we cannot find the result in literature explicitly.

\begin{proof}[Proof of Proposition \ref{prop:thin}]
Take $m=2n$ and $E=\overline{\Gamma}$. It follows that
\begin{eqnarray*}
E_k&\subset&\{z\in\mathbb{C}^n;\ |z'|^2+(\mathrm{Re}\,z_n)^2\leq(2^{-k}/C)^{2/\alpha},\ 2^{-k-2}\leq\mathrm{Im}\,z_n\leq{2^{-k+1}}\}\\
&=&\overline{B_{(2^{-k}/C)^{1/\alpha}}(0)}\times[2^{-k-2},2^{-k+1}]\subset\mathbb{R}^{2n-1}\times\mathbb{R}.
\end{eqnarray*}
Set $F_k:=\overline{B_{(2^{-k}/C)^{1/\alpha}}(0)}\times[2^{-k-2},2^{-k+1}]$. The affine mapping
\[T:\mathbb{C}^n\rightarrow\mathbb{C}^n,\ \ \ z\mapsto\left((2^{-k}/C)^{1/\alpha}z',(2^{-k}/C)^{1/\alpha}\mathrm{Re}\,z_n,2^{-k}\mathrm{Im}\,z_n\right)\]
maps the set $F:=\overline{B_1(0)}\times[1/4,2]$ onto $F_k$. For any $\varphi\in{C^1_0(\mathbb{C}^n)}$ with $\varphi|_{F_k}\geq1$, we have
\begin{eqnarray*}
\int_{F_k}|\nabla_x\varphi(x)|^2dx &=& \int_{F}\left|(\nabla_x\varphi)(Ty)\right|^2\left|\det{T}'(y)\right|dy\\
&\lesssim& 2^{2k/\alpha}\times2^{-(2n-1)k/\alpha}\times{2^{-k}}\int_F|\nabla_y(\varphi\circ{T})|^2\\
&=& 2^{-\{(2n-3)/\alpha+1\}k}\int_F|\nabla_y(\varphi\circ{T})|^2.
\end{eqnarray*}
Therefore,
\[\mathrm{Cap}\,(F_k)\lesssim2^{-\{(2n-3)/\alpha+1\}k}\]
and
\[\sum^\infty_{k=1}2^{k(2n-2)}\mathrm{Cap}\,(E_k)\leq\sum^\infty_{k=1}2^{k(2n-2)}\mathrm{Cap}\,(F_k)\lesssim\sum^\infty_{k=1}2^{-(2n-3)(1/\alpha-1)k}<\infty,\]
i.e., $\overline{\Gamma}$ is thin at the vertex when $n\geq2$.
\end{proof}


\begin{thebibliography}{99}
\bibitem{Blocki} Z. Blocki, The Complex Monge-Amp\`{e}re Operator in Pluripotential Theory, lecture notes, 2002, available at http://gamma.im.uj.edu.pl/~blocki.
 

\bibitem{Chen20} B.-Y. Chen, {\it Every bounded pseudoconvex domain with H\"{o}lder boundary is hyperconvex}, Bull. London Math. Soc. {\bf 53} (2021) 1009--1015.

\bibitem{Choquet} G. Choquet, {\it Theory of Capacities}, Ann. Inst. Fourier. {\bf 56} (1955) 131--295.

\bibitem{CM} M. Coltoiu and N. Mihalche, {\it Pseudoconvex domain on convex domains with singularities}, Compositio Math. {\bf 72} (1989), 241--247.


\bibitem{Demailly} J.-P. Demailly, {\it Mesures de Monge-Amp\`{e}re et mesures pluriharmoniques}, Math. Z. {\bf 194} (1987), 519--564.

\bibitem{FS} J. E. Fornaess and B. Stens{\o}nes, Lectures on Counterexamples in Several Complex Variables, Mathematical Notes, 33. Princeton University Press, Princeton, NJ; University of Tokyo Press, Tokyo, 1987.


\bibitem{KR} N. Kerzman and J.-P. Rosay, {\it Fonctions plurisousharmoniques d'exhaustion born\'{e}es et domaines taut},  Math. Ann. {\bf 257} (1981), no. 2, 171--184.  

\bibitem{Landkof} N. S. Landkof, Foundations of Modern Potential Theory, Die Grundlehren der mathematischen Wissenschaften, Band 180, Springer-Verlag, New York-Heidelberg, 1972.

\bibitem{LiNirenberg} Y. Y. Li and L. Nirenberg, {\it On the Hopf Lemma}, arXiv: 0709.3531v1, available at https://www.research gate.net/publication/1766306.

\bibitem{Mercer} P. Mercer, {\it A general Hopf lemma and proper holomorphic mappings between convex domains in $\mathbb{C}^n$}, Proc. Amer. Math. Soc. {\bf 119} (1993), no. 2, 573--578. 

\bibitem{Miller67} K. Miller, {\it Barriers on cones for uniformly elliptic operators}, Ann. Mat. Pura Appl. {\bf 76} (1967), 93--105.

\bibitem{Miller71} K. Miller, {\it Extremal barriers on cones with Phragm\`{e}n-Lindel\"{o}f theorems and other applications}, Ann. Mat. Pura Appl. {\bf 90} (1971), 297--329.

\bibitem{Oddson} J. K. Oddson, {\it On the boundary point priniple for elliptic equations in the plane}, Bull. Amer. Math. Soc. {\bf 74} (1968), 666--670.

\bibitem{Ransford} T. Ransford, Potential Theory in the Complex Plane, London Mathematical Society Student Texts, 28, Cambridge University Press, Cambridge, 1995.
\end{thebibliography}
\end{document}